\documentclass[11pt,a4paper]{amsart}
\makeatletter
\@namedef{subjclassname@1991}{2020 Mathematics Subject Classification}
\makeatother
\usepackage{cite}
\usepackage{xcolor}
\usepackage{amssymb}
\usepackage{hyperref}
\usepackage{amsfonts}
\usepackage{amsthm}
\usepackage{amsmath}
\usepackage{amscd}
\usepackage[latin2]{inputenc}
\usepackage{t1enc}
\usepackage[mathscr]{eucal}
\usepackage{indentfirst}
\usepackage{graphicx}
\usepackage{graphics}
\usepackage{pict2e}
\usepackage{epic}
\numberwithin{equation}{section}
\usepackage[margin=2.9cm]{geometry}
\usepackage{epstopdf}

\theoremstyle{plain}
\newtheorem{thm}{Theorem}[section]
\newtheorem{lem}[thm]{Lemma}

\newtheorem{prop}[thm]{Proposition}

\theoremstyle{definition}
\newtheorem{defn}[thm]{Definition}

\newtheorem{rem}[thm]{Remark}
\newtheorem{?}[thm]{Problem}

\theoremstyle{definition}
\newtheorem*{nt*}{Notation}

\theoremstyle{plain}
\newtheorem{asu}{Assumption}[section]

\newcommand{\inner}[2]{\left\langle {#1}, {#2} \right\rangle}

\newcommand {\R} {\mathbb R}
\newcommand {\Z} {\mathbb Z}

\newcommand {\calh} {\mathcal H}
\newcommand {\dla} {\Delta}

\newcommand {\tap} {\text{Sol}(U,\Omega)}

\begin{document}

\title[]{ Second-order dynamical systems with a smoothing effect for solving paramonotone variational inequalities}

\author{Pham Viet Hai and Trinh Ngoc Hai}
\address[]{Faculty of Mathematics and Informatics, Hanoi University of Science and Technology, 1 Dai Co Viet, Hanoi, Vietnam.}%
\email{hai.phamviet@hust.edu.vn; hai.trinhngoc@hust.edu.vn}



\begin{abstract}
In this paper, we propose a second-order dynamical system with a smoothing effect for solving paramonotone variational inequalities. Under standard assumptions, we prove that the trajectories of this dynamical system converges to a  solution of the variational inequality problem. A time discretization of this dynamical system provides an iterative inertial projection-type method. Our result generalizes and improves the existing results. Some numerical examples are given to confirm the theoretical results and illustrate the effectiveness of the proposed algorithms.
\end{abstract}

\maketitle
\section{Introduction}
\subsection{Variational inequality}
Let $\Omega\subset\R^d$ be a nonempty, closed, convex subset and let $U:\Omega\to\Omega$ be an operator. The variational inequality (VI) of the operator $U$ on $\Omega$ is
\begin{gather}\label{proVI} \tag{VI($U,\Omega$)}
    \text{to find $x_\star\in\Omega$ such that $\inner{U(x_\star)}{a-x_\star}\geq 0\quad\forall a\in\Omega$.}
\end{gather}
The solution set of the problem above is denoted as $\tap$. This problem plays a central role in optimization and has been in-depth investigated in the literature \cite{Bao,CensorGibali,Cruz,Iiduka2,Iiduka1,BelloIusem,HaiOPTL,KhanhVuong,Kinderlehrer,MalitskySemenov,SolodovSvaiter}.

When $U$ is Lipschitz continuous and satisfies some monotonicity condition, numerous algorithms have been introduced for solving \ref{proVI} \cite{AnhVinhDOI,CensorGibali,Facchinei,Quoc1} However, if $U$ is  not Lipschitz continuous, there are very few  methods for solving this problem. In this work, we consider problem \ref{proVI} when $U$ is paramonotone but is not Lipschitz continuous. Under such an assumption, Hai \cite{HaiDS}  proposed the first order dynamical system
\begin{gather}
\begin{cases}\label{hai26}
    x'(t)=\delta(t)\left\{P_{\Omega}\left(x(t)-\frac{\alpha(t)}{\max\{1,\|U(x(t))\|\}}U(x(t))\right)-x(t)\right\},\\
    x(t_0)=x_0\in\Omega,
\end{cases}
\end{gather}
where $\alpha,\delta:[0,\infty)\to (0,\infty)$. Assuming  $\int_0^\infty \alpha (t)dt =\infty $ and $\int_0^\infty \alpha(t)^2dt <\infty$, the author proved that $x(t)$ converges to a solution of \eqref{proVI}.
When the operator $U$ is strongly pseudo-monotone and Lipschitz continuous, Vuong \cite{VuongSIAM} proposes the second order dynamical system
\begin{gather}\label{202402122110}
\begin{cases}
    x''(t)+\alpha_1(t)x'(t)=\delta(t)\left\{P_{\Omega}\left(x(t)-\alpha_0U(x(t))\right)-x(t)\right\},\\
    x(t_0)=x_0\in\calh,\quad x'(t_0)=x_1\in\calh,
\end{cases}    
\end{gather}
where $\alpha_0,\alpha_1,\delta:[0,\infty)\to (0,\infty)$. A finite-difference scheme for \eqref{202402122110} with respect to the time variable $t$, gives rise to the following
iterative scheme:
\begin{gather*}
\begin{cases}
    y_k=x_k+\theta(x_k-x_{k-1}),\\
    x_{k+1}=(1-\rho)y_k+\rho P_{\Omega}(y_k-\lambda Uy_k).
\end{cases}
\end{gather*}
The second order dynamical systems \eqref{202402122110}  has a close connection with
the heavy-ball method and numerical methods with inertial effect. However, it requires more strict conditions than  \eqref{hai26} does. 
As a natural extension, it is interesting to study possibility of applying the second order dynamical system to solve paramonotone variational inequalities. This is the main aim of our paper. To do this, we propose the following dynamical system
\begin{gather}\label{hai27}
    \begin{cases}
    y(t)\triangleq P_\Omega\left(x(t)+\lambda(t)x'(t)-\frac{\alpha_0(t)}{\max\{1,\|U(x(t)+\lambda(t)x'(t))\|\}}U(x(t)+\lambda(t)x'(t))\right),\\
        x''(t)+\alpha_1(t)x'(t)=\delta(t)[y(t)-x(t)],
    \end{cases}
\end{gather}
where $\alpha_0,\alpha_1,\delta:[0,\infty)\to (0,\infty)$ and $\lambda:[0,\infty)\to [0,\infty)$. We not only improve \eqref{hai26} to second order setting, but also smooth it out. Namely, motivated by the idea in \cite{Alecsa}, we add a damping term $\lambda (t) x'(t)$ to the argument of the projection. According to \cite{Alecsa}, this term annihilates the oscillations for the errors $\|x(t)-x_\star\|$, where $x_\star$ is a solution of \eqref{proVI}. Under the assumption that $U$ is paramonotone and continuous, we prove the convergence of trajectories generated by \eqref{hai27}.  It is worth mentioning that, to obtain the convergence of \eqref{hai27}, we need to prove that trajectories of this dynamical system do not leave set $\Omega$, which is an interesting result itself.

Further,  by discretizing the dynamical system \eqref{hai27}, we obtain a novel iterative algorithms with an  inertial effect. The convergence of this algorithm is established under the similar assumptions as that of \eqref{hai27}. 

\subsection{Related works}
There is a growing interest in using ordinary differential equations (ODEs) to design algorithms for optimization problems. The earliest appearance of ODEs in this research direction may be the continuous gradient method
\begin{gather*}
    z'(t)=-\nabla f(z(t))
\end{gather*}
in connection with the unconstrained minimization problem of the differentiable convex function $f$. In order to enrich the methods solving, researchers have increased the order of ODEs with the desire that equations obtained exhibit a better performance. Polyak \cite{Polyak} proposed the Heavy Ball with Friction method to accelerate the gradient descent method. Polyak's method regards an inertial system with a fixed viscous damping coefficient 
\begin{equation} \label{HB}
    z''(t) +\delta z'(t)=-\nabla f (z(t)).
\end{equation}
The system \eqref{HB} was extended for constrained optimization as well as co-coercive operator (see \cite{AA}). Later, Bo\c t and Csetnek \cite{Bot} introduced the second order dynamical system enhanced by the variable viscous damping coefficient 
$$
z''(t) +\delta(t)z'(t)=-\theta(t)A(z(t)),
$$
where the operator $A$ is co-coercive in a real Hilbert space. One finding in \cite{Bot} states that it was possible to prove a weak convergence to the zero of the operator $A$. These works have motivated the study of dynamical systems for solving monotone inclusions and optimization problems (\cite{Alecsa, zbMATH06522738, ACR, zbMATH07194541}). The survey \cite{zbMATH07344788} is a good source providing the recent progresses in this research direction.

\subsection{Organization of the paper}
The remainder of this paper is structured as follows. Section 2 collects some  concepts and results which
will be frequently used in this paper.  In Section 3, we prove the convergence of the dynamical system in continuous time. Section 4 is devoted to convergence of an inertial iterative projection-type algorithm, obtained via discretization of the corresponding dynamical system. Section \ref{sec-exam} gives some numerical examples to confirm the theoretical results and illustrate the effectiveness of the proposed algorithms.

\section{Preliminaries}
Let $X_{\geq a}\triangleq\{t\in X:t\geq a\}$ and $X_{>a}\triangleq\{t\in X:t>a\}$, where $X\subseteq\R$. For $p\geq 1$, the space $L^p$ consists of functions that are $p$-power Lebesgue integrable over $\R_{\geq 0}$. $L^\infty$ includes functions that are essentially bounded measurable over $\R_{\geq 0}$. The space $\ell^p$ is a type of sequence space that consist of sequences of numbers that are $p$-power summable over $\Z_{\geq 0}$. The space $\ell^\infty$ contains sequences of numbers that are bounded over $\Z_{\geq 0}$. 

\subsection{Monotone operators}
For the purpose of the paper, it is to recall some terminologies.
\begin{defn}
    An operator $U:\Omega\subseteq\calh\to\calh$ is called
    \begin{enumerate}
        \item \emph{monotone} if
        \begin{gather*}
            \inner{U(a)-U(b)}{a-b}\geq 0\quad\forall a,b\in\Omega.
        \end{gather*}
        \item \emph{paramonotone} if it is monotone and for every $a,b\in\Omega$ we have
        \begin{gather*}
            \inner{U(a)-U(b)}{a-b}=0\iff U(a)=U(b).
        \end{gather*}
    \end{enumerate}
\end{defn}

Lemma \ref{lem-tap} helps us determine when an element belongs to the solution set $\tap$.
\begin{lem}[\cite{HaiDS}]\label{lem-tap}
    Let $U:\Omega\to\Omega$ be a paramonotone operator. For every $a\in\tap$ and for every $b\in\Omega$, it holds that
    \begin{gather*}
        \inner{U(b)}{b-a}=0\Longrightarrow b\in\tap.
    \end{gather*}
\end{lem}

Throughout the paper, we study Problem \ref{proVI} under the following assumption.
\begin{asu}\label{asu-U}
    Assume that (i) $\Omega$ is a nonempty, closed, convex subset of $\R^d$; (ii) the set $\tap\ne\emptyset$; (iii) the operator $U$ is paramonotone and continuous on  $\Omega$.
\end{asu}

\subsection{Ordinary differential equations}
The following lemmas can be found in the paper \cite{Abbas}.
\begin{lem}[{\cite{Abbas}}]\label{lem-main-0}
Suppose that $u:\R_{\geq 0}\to\R$ is locally absolutely continuous and bounded below and $w\in L^1(\R_{\geq 0})$. If for almost every $t\in\R_{\geq 0}$
\begin{gather*}
   u'(t)\leq w(t),
\end{gather*}
then there exists $\lim\limits_{t\to\infty}u(t)$.
\end{lem}

\begin{lem}[{\cite{Abbas}}]\label{lem-main-1}
Let $1\leq p<\infty$ and $1\leq q\leq\infty$. Let $u\in L^p(\R_{\geq 0})$, $w\in L^q(\R_{\geq 0})$ such that $u:\R_{\geq 0}\to\R_{\geq 0}$ is locally absolutely continuous and $w:\R_{\geq 0}\to\R$. If for almost every $t\in\R_{\geq 0}$
\begin{gather*}
u'(t)\leq w(t),
\end{gather*}
then $\lim\limits_{t\to\infty}u(t)=0$.
\end{lem}


Let $\alpha_0,\alpha_1,\delta:[0,\infty)\to (0,\infty)$ and $\lambda:[0,\infty)\to [0,\infty)$. Consider the ordinary differential equation
\begin{gather}\label{202401310822}
    \begin{cases}
        x''(t)+\alpha_1(t)x'(t)=\delta(t)[y(t)-x(t)],\\
        x(t_0)=x_0\in\Omega,\quad x'(t_0)=\frac{1}{4}\alpha_1(t_0) (x_1-x_0), \text{where } x_1\in \Omega.
    \end{cases}
\end{gather}
Here
\begin{gather}
\label{202402250847}
y(t)\triangleq P_\Omega\left(x(t)+\lambda(t)x'(t)-\frac{\alpha_0(t)}{\max\{1,\|U(x(t)+\lambda(t)x'(t))\|\}}U(x(t)+\lambda(t)x'(t))\right).
\end{gather}

The solution of dynamical system \eqref{202401310822} is understood in the following sense.
\begin{defn}
A function $x$ is called a \emph{strong global solution} of equation \eqref{202401310822} if it holds:
\begin{enumerate}
    \item The functions $x,x',x'':[t_0,\infty)\to\calh$ is locally absolutely continuous; in other words, absolutely continuous on each interval $[a,b]$ for $b>a>t_0$.
    \item {\scriptsize$x''(t)+\alpha_1(t)x'(t)=\delta(t)\left\{P_\Omega\left(x(t)+\lambda(t)x'(t)-\frac{\alpha_0(t)}{\max\{1,\|U(x(t)+\lambda(t)x'(t))\|\}}U(x(t)+\lambda(t)x'(t))\right)-x(t)\right\}$} for almost every $t\geq t_0$.
    \item $x(t_0)=x_0\in\Omega,x'(t_0)=\frac{1}{4}\alpha_1(t_0) (x_1-x_0)\in\Omega$.
\end{enumerate}
\end{defn}

\begin{prop}[Equivalent form]
Equation \eqref{202401310822} is equivalent to the system $\omega'(t)=G(t,\omega(t))$, where $G:[t_0,\infty)\times\calh\times\calh\to\calh\times\calh$ is defined by
{\small
\begin{gather*}
    G(t,\omega_1,\omega_2)=\bigg(\omega_2(t),-\alpha_1(t)\omega_2(t)\\
    +\delta(t)\left\{P_\Omega\left(\omega_1(t)+\lambda(t)\omega_2(t)-\frac{\alpha_0(t)}{\max\{1,\|U(\omega_1(t)+\lambda(t)\omega_2(t))\|\}}U(\omega_1(t)+\lambda(t)\omega_2(t))\right)-\omega_1(t)\right\}\bigg),
\end{gather*}
}
where $\omega=(\omega_1,\omega_2)\in\calh\times\calh.$
\end{prop}
\begin{proof}
The conclusion follows from doing the change of variables
\begin{gather*}
    (\omega_1(t),\omega_2(t))=\left(x(t),x'(t)\right).
\end{gather*}
\end{proof}

The proof of the following result is left to the reader as it is similar to \cite{VuongSIAM}.
\begin{prop}[Existence and uniqueness of a strong global solution]
Consider the dynamical system \eqref{202401310822}, where $\alpha_0,\alpha_1,\delta:[t_0,\infty)\to\R_{>0}$ are locally integrable functions in the Lebesgue sense and the operator $U$ is Lipschitz. Then for each $x_0\in\Omega,x_1\in\Omega$, there exists a unique strong global solution of the dynamical system \eqref{202401310822}.    
\end{prop}

The following result shows that $x(t)\in\Omega$ for every $t\geq t_0.$

\begin{lem}\label{hai1}
Let $x(t)$ be the solution of the dynamical system 
\begin{equation*}
\begin{cases}
&x({t_0})\in \Omega,\\
&x'(t) =\delta (t) \left( y(t) -x(t) \right),
\end{cases}
\end{equation*}
where $y:[t_0;\infty) \to \Omega$ and $\delta: [t_0;\infty) \to (0;\infty)$ are  continuous functions satisfying $y(t)\in \Omega$ for all $t\geq t_0$.
 Then, we have $x(t)\in \Omega$ for all $t\in [t_0,\infty)$.
\end{lem}
\begin{proof}
For simplicity's sake, let $t_0=0$.
It holds that
\begin{equation*}
x'(t)+\delta (t) x(t) =\delta (t)  y(t) .
\end{equation*}
Hence,
\begin{align*}
x(t)= &e^{-\int_0^t\delta (u)du }\left( \int_0^t \delta (s) y(s) e^{\int_0^s\delta (u)du } \text{ds}+ x(0)   \right)\\
=&\left( 1-e^{-\int_0^t\delta (u)du } \right) \frac {1}{e^{\int_0^t\delta (u)du }-1}  \int_0^t \delta (s) y(s) e^{\int_0^s\delta (u)du } \text{ds} + e^{-\int_0^t\delta (u)du } x(0). 
\end{align*}
Since $\Omega$ is convex and $x(0) \in \Omega$, it is sufficient to show that 
\begin{equation*}
 \frac {1}{e^{\int_0^t\delta (u)du }-1}  \int_0^t \delta (s) y(s) e^{\int_0^s\delta (u)du } \text{ds} \in \Omega\quad \forall t>0.
\end{equation*}
We have 
\begin{equation*}
 \frac {1}{e^{\int_0^t\delta (u)du }-1} \int_0^t \delta (s) y(s) e^{\int_0^s\delta (u)du } \text{ds} = \lim_{n\to \infty}\sum_{i=0}^{n-1}\frac {  e ^{ \int_0^ {\frac {it}{n}}\delta (u)du }}{e^{\int_0^t\delta (u)du }-1}    y\left(\frac {it}{n}\right) \delta \left(\frac {it}{n}\right) \frac{t}{n}.
\end{equation*}
Denote $$\xi(t):= \frac{d}{dt} \left(  e ^{ \int_0^t\delta (u)du }  \right)=  e ^{ \int_0^t\delta (u)du } \delta (t).$$
Then,
\begin{equation*}
\lim\limits_{n\to \infty} \sum_{i=0}^{n-1} e ^{ \int_0^ {\frac {it}{n}}\delta (u)du }     \delta \left(\frac {it}{n}\right) \frac{t}{n} 
=\lim\limits_{n\to \infty}  \sum_{i=0}^{n-1} \xi \left(\frac {it}{n}\right) \frac t n= \int_0^t \xi(u) du=e^{\int_0^t\delta (u)du }-1.
\end{equation*}
We found that
\begin{itemize}
	\item[$\bullet$] $ \frac {  e ^{ \int_0^ {\frac {it}{n}}\delta (u)du }}{e^{\int_0^t\delta (u)du }-1}      \delta \left(\frac {it}{n}\right) \frac{t}{n} \in (0;1) $;
	\item[$\bullet$] $ 
	\lim_{n\to \infty} \sum_{i=0}^{n-1}\frac {  e ^{ \int_0^ {\frac {it}{n}}\delta (u)du }}{e^{\int_0^t\delta (u)du }-1}  \delta \left(\frac {it}{n}\right) \frac{t}{n}=1;
	$
	\item[$\bullet$]  $y\left(\frac {it}{n}\right) \in \Omega $;
	\item[$\bullet$]  $\Omega$ is closed and convex.
\end{itemize}
So we have
\begin{equation*}
 \frac {1}{e^{\int_0^t\delta (u)du }-1} \int_0^t \delta (s) y(s) e^{\int_0^s\delta (u)du } \text{ds} = \lim_{n\to \infty}\sum_{i=0}^{n-1}\frac {  e ^{ \int_0^ {\frac {it}{n}}\delta (u)du }}{e^{\int_0^t\delta (u)du }-1}    y\left(\frac {it}{n}\right) \delta \left(\frac {it}{n}\right) \frac{t}{n} \in \Omega.
\end{equation*}
\end{proof}

\begin{lem}\label{lemmahai2}
Consider the dynamical system 
	\begin{equation}\label{syshai2}
		\begin{cases}
			&x(t_0)=x_0\in \Omega,\\
			&x'(t_0)=\frac{1}{4}\alpha_1 (t_0) (x_1-x_0), \text{ where } x_1\in \Omega,\\
			&x''(t)+\alpha_1 (t) x'(t) =\delta (t) \left[ y(t) -x(t) \right].
		\end{cases}
	\end{equation}
 Here, $y:[t_0;\infty) \to \Omega$ and $\delta, \alpha_1: [t_0;\infty) \to (0;\infty)$ are locally absolutely continuous functions satisfying $y(t)\in \Omega$ for all $t\geq t_0$ and
\begin{gather}\label{202403022125}
\delta(t)<\frac{1}{4}(\alpha_1(t)^2+2\alpha_1'(t))\quad\forall t\geq t_0.
\end{gather}
	Then, $x(t)\in \Omega$ for all $t\geq t_0$.
	\end{lem}
\begin{proof}
We prove that the dynamical system \eqref{syshai2} is equivalent to the following
\begin{gather}\label{202403120843}
    \begin{cases}
x'(t)=\gamma(t) (u(t)-x(t)),\\	
u'(t)=\mu(t) (y(t)-u(t)),
\end{cases}
\end{gather}
where 
\begin{gather}\label{202403120858}
\begin{cases}
\gamma'(t)+\alpha_1 (t) \gamma(t) =\gamma(t)^2 +\delta(t),\\
\gamma(t_0) =\frac 1 4 \alpha_1(t_0),    
\end{cases}
\end{gather}
and
\begin{gather}
\label{2024031208581} \mu(t)=\frac{\delta(t)}{\gamma(t)}.
\end{gather}
Indeed, the implication \eqref{syshai2}$\Longrightarrow$\eqref{202403120843} is proven as follows. Let $u$ be the solution of the equation
\begin{gather*}
\begin{cases}
u'(t)=\mu(t) (y(t)-u(t)),\\
u(t_0)=x_1.
\end{cases}
\end{gather*}
A direct computation gives
\begin{gather*}
(x'+\gamma x)'=(\gamma u)',
\end{gather*}
which implies, after integrating over $[t_0,t]$, that
\begin{gather*}
x'(t)+\gamma(t)x(t)-x'(t_0)-\gamma(t_0)x(t_0)=\gamma(t)u(t)-\gamma(t_0)u(t_0).    
\end{gather*}
Since $x'(t_0)+\gamma(t_0)x(t_0)=\gamma(t_0)u(t_0)$, we get \eqref{202403120843}.

Conversely, suppose that the pair $(x,u)$ satisfies \eqref{syshai2}. We have
\begin{gather*}
x''=[\gamma(u-x)]'=\gamma'(u-x)+\gamma(u'-x')\\
=\gamma'(u-x)+\gamma\mu(y-u)-\gamma x'\\
=-\alpha_1x'+\delta(y-x).
\end{gather*}

Next is to show that \eqref{202403120858} has a solution $\gamma (t)$ defined on $[t_0;\infty)$. The existence of a local solution of this equation follows from the Picard-Lindelof theorem. It remains to prove that $\gamma (t)$ does not tend to infinity  at any finite $t_b \in [0;\infty).$  Denote $z(t)= \gamma (t)-\frac 1 2 \alpha_1 (t)$, $ A(t)= \delta (t) -\frac 1 4 \alpha_1(t)^2 -\frac{1}{2} \alpha_1' (t), $ \eqref{202403120858}  becomes
$$
z'(t)= z(t)^2 +A(t).
$$
We have $z'(t) \geq A(t)$, or
$$
z(t) -z(t_0) \geq \int_{t_0}^t A(s)ds   \quad \forall t\geq t_0.
$$
Thus, $\gamma (t)$ does not tend to $-\infty$ at $t_b$. Next, we prove that $z(t) <0$ for all $t\geq t_0$. Suppose, on the contrary that there exists $t_1 >t_0$ such that $z(t_1) =0$. Take a small $\epsilon >0$.  Since $z(t_0)<0$, without loss of generality, we may assume that $z(t) <0$ for all $t\in [t_0;t_1-\epsilon]$. Since $z'(t) \leq z(t)^2$ and $z(t) \neq 0$ for all $t \in [t_0;t_1-\epsilon]$, we have
$$
\frac{1 }{z(t_0)}-\frac{1 }{z(t_1-\epsilon)}   =\int_{t_0}^{t_1-\epsilon} \frac{ z'(t)}{z(t)^2}dt \leq \int_{t_0}^{t_1-\epsilon} dt <t_1-t_0.
$$
Letting $\epsilon \to 0^{+}$, we obtain a contradiction.

So, there exist continuous function $\mu(t)$ and $\gamma (t)$ satisfying \eqref{202403120843}-\eqref{2024031208581}.
It follows that
$$
\gamma (t) =\exp \left( -\int_{t_0}^t \alpha_1 (u)du \right) \left( \int_{t_0}^{t} \exp \left( \int_{t_0}^u \alpha_1 (s)ds \right) (\gamma(u)^2+\delta (u))du +\gamma ({t_0})   \right).
$$
Thus, $\gamma (t)$ and $\mu (t) \geq {t_0}$ for all $t\geq {t_0}$.	Moreover, $u({t_0})=x_1   \in \Omega$. According to Lemma \ref{hai1},  $u(t)$ and $x(t)$ belong to $\Omega$ for all $t\geq {t_0}.$
\end{proof}
\begin{rem}
The result of Lemma \ref{hai1} cannot be extended to the 2-order dynamical system without  condition \eqref{202403022125}. Indeed, consider \eqref{syshai2} with $\Omega=[1;2] \subset \mathbb R $, $t_0=0$, $x_0=x_1=2$, $\alpha_1(t)=\delta (t)=2$, $y(t)=1$ for all $t\geq 0$. This dynamical system has a unique solution $x(t)=1+e^{-t}(\cos t+\sin t)$. But $x(t) \notin \Omega$ when $t\in (\pi; \frac 3 2 \pi)$ .
\end{rem}

\subsection{Difference equation}
In the section, we give the discrete counterpart of the dynamical system \eqref{202401310822}. To that aim, we recall the operation of forward difference and its properties used in the convergence analysis.
For $z:\Z\to\calh$ and $\kappa\in\Z_{\geq 1}$, we denote
\begin{gather*}
z^\dla(n)\triangleq z(n+1)-z(n),\quad z^\nabla(n)\triangleq z(n)-z(n-1),\\
z^{\dla\nabla}\triangleq (z^\dla)^\nabla,\quad z^{\nabla\dla}\triangleq (z^\nabla)^\dla.
\end{gather*}

\begin{prop}\label{rem20220120}
Let $f,g,h:\Z\to\R^d$. Always have
\begin{gather*}
\inner{h}{g}^\dla(n)=\inner{h^\dla(n)}{g(n)}+\inner{h(n)}{g^\dla(n)}+\inner{h^\dla(n)}{g^\dla(n)},\\
\inner{h}{g}^\nabla(n)=\inner{h^\nabla(n)}{g(n)}+\inner{h(n)}{g^\nabla(n)}-\inner{h^\nabla(n)}{g^\nabla(n)},\\
z^{\dla\nabla}(n)=z^{\nabla\dla}(n)=z(n+1)-2z(n)+z(n-1)=z^\dla(n)-z^\nabla(n).
\end{gather*}
\end{prop}

Consider the difference equation, which is the discrete version of \eqref{202401310822}: 
\begin{gather}\label{202402061434}
\begin{cases}
z^{\dla\nabla}(n)+\beta_1(n)z^{\nabla}(n)=\xi(n)[w(n)-z(n)],\\
z(0)=z_0\in\Omega,z(1)=z_1\in\Omega.    
\end{cases}  
\end{gather}
where $\xi,\beta_1,\beta_0:\Z\to[0;\infty)$, $\eta:\Z\to\R$ and
\begin{gather}\label{202402250850}
w(n)\triangleq P_\Omega\left(z(n)+\eta(n)z^{\nabla}(n)-\frac{\beta_0(n)}{\max\{1,\|U(z(n)+\eta(n)z^{\nabla}(n))\|\}}U(z(n)+\eta(n)z^{\nabla}(n))\right).
\end{gather}

\begin{prop}[Equivalent form]
Equation \eqref{202402061434} has an equivalent form
\begin{gather}\label{al2}
    z(n+1)=[2-\beta_1(n)-\xi(n)]z(n)+[\beta_1(n)-1]z(n-1)+\xi(n)w(n).
\end{gather}
\end{prop}
\begin{proof}
The proof makes use of Proposition \ref{rem20220120}.
\end{proof}

\begin{rem}
The numerical scheme \eqref{al2} can be re-written as 
\begin{gather*}
z(n+1)=z(n)+[1-\beta_1(n)][z(n)-z(n-1)]+\xi(n)\left\{w(n)-z(n)\right\}, 
\end{gather*}
which is an algorithm. 
\end{rem}

\begin{prop}\label{202402092014}
    Consider the difference equation \eqref{202402061434}, where $1\leq\beta_1(n)\leq\beta_1(n)+\xi(n)\leq 2$ for every $n\in\Z_{\geq 0}$. Then $z(n)\in\Omega$ for every $n\in\Z_{\geq 0}$.
\end{prop}
\begin{proof}
Note that the assumption of the proposition gives
\begin{gather*}
    0\leq 2-\beta_1(n)-\xi(n)\leq 1,\\
    0\leq\beta_1(n)-1\leq 1,\\
    0\leq\xi(n)\leq 1,\\
    [2-\beta_1(n)-\xi(n)]+[\beta_1(n)-1]+\xi(n)=1.
\end{gather*}
Thus, the proof is complete by using the induction argument on $n\in\Z_{\geq 0}$ and using \eqref{al2}.
\end{proof}

\section{Convergence analysis in continuous time}
With all preparation in place we can now conduct on the convergence research of the dynamical system \eqref{202401310822}. Our analysis is done with the help of the functions:
\begin{gather}\label{202402250851}
    v_{x_*}(t)\triangleq\frac{1}{2}\|x(t)-x_*\|^2,\quad b(t)\triangleq\frac{1}{2}\|x'(t)\|^2.
\end{gather}
Here $x_*\in\tap.$

\begin{lem}
Suppose that condition \eqref{202403022125} holds. Let $y(t)$ be the function given by \eqref{202402250847}. Then it holds that
\begin{gather}
\nonumber\|y(t)-x_*\|^2\\
\label{202401310841}\leq 2v_{x_*}(t)+2\lambda(t)^2b(t)+2\lambda(t)v_{x_*}'(t)+\alpha_0(t)^2\\
\nonumber-\frac{2\alpha_0(t)}{\max\{1,\|U(x(t)+\lambda(t)x'(t))\|\}}\inner{U(x(t)+\lambda(t)x'(t))}{x(t)+\lambda(t)x'(t)-x_*}\\
\label{202401310842}\leq 2v_{x_*}(t)+2\lambda(t)^2b(t)+2\lambda(t)v_{x_*}'(t)+\alpha_0(t)^2
\end{gather}
and
\begin{gather}
\label{202401310830}\|y(t)-x(t)\|\leq\lambda(t)\|x'(t)\|+\alpha_0(t).
\end{gather}
\end{lem}
\begin{proof}
We have
\begin{gather}
    \nonumber\|x(t)-x_*+\lambda(t)x'(t)\|^2\\
    \nonumber=\|x(t)-x_*\|^2+\lambda(t)^2\|x'(t)\|^2+2\lambda(t)\inner{x'(t)}{x(t)-x_*}\\
    \nonumber=2v_{x_*}(t)+2\lambda(t)^2b(t)+2\lambda(t)v_{x_*}'(t).
\end{gather}
Since $x(t)\in\Omega$ by Lemma \ref{lemmahai2}, we have $P_\Omega(x(t))=x(t)$. We estimate
{\small
\begin{gather}
\nonumber\|y(t)-x_*\|^2\\
\nonumber\leq\left\|x(t)+\lambda(t)x'(t)-x_*-\frac{\alpha_0(t)}{\max\{1,\|U(x(t)+\lambda(t)x'(t))\|\}}U(x(t)+\lambda(t)x'(t))\right\|^2\\
\nonumber=\|x(t)+\lambda(t)x'(t)-x_*\|^2+\left(\frac{\alpha_0(t)}{\max\{1,\|U(x(t)+\lambda(t)x'(t))\|\}}\right)^2\|U(x(t)+\lambda(t)x'(t))\|^2\\
\nonumber-\frac{2\alpha_0(t)}{\max\{1,\|U(x(t)+\lambda(t)x'(t))\|\}}\inner{U(x(t)+\lambda(t)x'(t))}{x(t)+\lambda(t)x'(t)-x_*},
\end{gather}
}
which proves \eqref{202401310841} and \eqref{202401310842}. Meanwhile,
{\small
\begin{gather}
\nonumber\|y(t)-x(t)\|\\
\nonumber=\left\|P_\Omega\left(x(t)+\lambda(t)x'(t)-\frac{\alpha_0(t)}{\max\{1,\|U(x(t)+\lambda(t)x'(t))\|\}}U(x(t)+\lambda(t)x'(t))\right)-P_\Omega(x(t))\right\|\\
\nonumber\leq\left\|\lambda(t)x'(t)-\frac{\alpha_0(t)}{\max\{1,\|U(x(t)+\lambda(t)x'(t))\|\}}U(x(t)+\lambda(t)x'(t))\right\|,
\end{gather}
}   
which proves \eqref{202401310830}.
\end{proof}

We study the dynamical system \eqref{202401310822}, under conditions \eqref{202402201440}-\eqref{202401310813}.
\begin{asu}\label{asu-coefficients}
    Assume that there exist constants $C_1,C_2>0$ such that
\begin{gather}
\label{202401310811}2\alpha_1(t)-2\delta(t)\lambda(t)-C_2\delta(t)\geq C_1\quad\forall t\geq t_0.
\end{gather}
And assume that
    \begin{gather}
    \nonumber\text{Condition \eqref{202403022125} holds}.\\
    \label{202402201440}\text{The functions $\alpha_0,\delta:[t_0,\infty)\to\R_{>0}$ are locally integrable}.\\
    \label{202401310810-o}\text{The functions $\alpha_1,\delta:[t_0,\infty)\to\R_{>0}$ and $\lambda:[t_0,\infty)\to\R_{\geq 0}$ are locally absolutely continuous}.\\
    \label{202402081959}\text{The function $\lambda(t)$ is bounded above.}\\
    \label{202403120830}\lambda(t)\gamma(t)\leq 1\quad\forall t\geq t_0.\\
    \label{202401310810}\frac{d}{dt}[\alpha_1(t)-\delta(t)\lambda(t)]\leq 0\quad\forall t\geq t_0.\\
\label{202401310812}\int\limits_{t_0}^\infty\delta(t)\alpha_0(t)^2\,dt<\infty.\\
\label{202401310813}\int\limits_{t_0}^\infty\delta(t)\alpha_0(t)\,dt=\infty.
    \end{gather}
\end{asu}

\begin{rem}
Condition \eqref{202401310810} can be replaced by
\begin{gather*}
\alpha_1'(t)\leq 0\leq\min\{\delta'(t),\lambda'(t)\}.    
\end{gather*}
Indeed, we have
\begin{gather*}
\frac{d}{dt}[\alpha_1(t)-\delta(t)\lambda(t)]=\alpha_1'(t)-\delta'(t)\lambda(t)-\delta(t)\lambda'(t)\leq 0.    
\end{gather*}
\end{rem}

\begin{rem}
Under Assumption \ref{asu-coefficients}, we observe some remarks.
\begin{enumerate}
\item By \eqref{202402081959}, we denote
\begin{gather}\label{202402201529}
    \lambda_{\sup}\triangleq\sup\{\lambda(t):t\geq t_0\}.
\end{gather}
\item By \eqref{202403120830}, we prove
\begin{gather*}
    x(t)+\lambda(t)x'(t)\in\Omega\quad\forall t\geq t_0.
\end{gather*}
Indeed, using the equivalent form \eqref{202403120843}, we have
\begin{gather*}
x(t)+\lambda(t)x'(t)=\lambda(t)\gamma(t)u(t)+(1-\lambda(t)\gamma(t))x(t)\in\Omega\quad\forall t\geq t_0.    
\end{gather*}
    \item By \eqref{202401310811} and \eqref{202401310810} we have
\begin{gather*}
    C_2\delta(t)\leq 2\alpha_1(t)-2\delta(t)\lambda(t)\leq 2\alpha_1(t_0)-2\delta(t_0)\lambda(t_0)\leq 2\alpha_1(t_0),
\end{gather*}
and so
\begin{gather}\label{202401311436}
    \delta(t)\leq\frac{2\alpha_1(t_0)}{C_2}.
\end{gather}
\item Again using \eqref{202401310811} and \eqref{202401310810}, we get
\begin{gather}
    \nonumber\alpha_1(t)=[\alpha_1(t)-\delta(t)\lambda(t)]+\delta(t)\lambda(t)\\
    \nonumber\leq\alpha_1(t_0)+\delta(t)\lambda(t)\\
    \label{202402090927}\leq\alpha_1(t_0)+\frac{2}{C_2}\alpha_1(t_0)\lambda_{\sup}.
\end{gather}
\end{enumerate}
\end{rem}

\subsubsection{Theoretical analysis}
\begin{thm}\label{thm202402061556}
Consider the dynamical system \eqref{202401310822} under Assumptions \ref{asu-U} and \ref{asu-coefficients}. Moreover, suppose that  this dynamical system admits a global solution $x(t)$. Then the following conclusions hold.
\begin{enumerate}
    \item[(i)] $\|x''\|\in L^2$.
    \item[(ii)] $\lim\limits_{t\to\infty}x'(t)=0$.
    \item[(iii)] $\lim\limits_{t\to\infty}x(t)=x_\star\in\tap.$
\end{enumerate}
\end{thm}
\begin{proof}
Our analysis is done with the help of the functions in \eqref{202402250851}. For the sake of the reader, we divide the proof into the steps.

{\bf Step 1:} Prove that
\begin{gather}\label{202402020817}
    \|x'\|\in L^\infty\cap L^2.
\end{gather}
A direct computation shows
\begin{gather*}
    b'(t)+2\alpha_1(t)b(t)=\inner{x''(t)+\alpha_1(t)x'(t)}{x'(t)}\\
    =\delta(t)\inner{y(t)-x(t)}{x'(t)}\quad\text{(by \eqref{202401310822})}\\
    \leq\delta(t)\|y(t)-x(t)\|\cdot\|x'(t)\|\\
    \leq\delta(t)[\lambda(t)\|x'(t)\|+\alpha_0(t)]\cdot\|x'(t)\|\quad\text{(by \eqref{202401310830})},
\end{gather*}
which implies, by using the Cauchy-Schwarz inequality, that
\begin{gather*}
b'(t)+2\alpha_1(t)b(t)
\leq 2\delta(t)\lambda(t)b(t)+\delta(t)\left[\frac{\alpha_0(t)^2}{2C_2}+C_2b(t)\right].
\end{gather*}
Then
\begin{gather}
\nonumber\frac{1}{2C_2}\delta(t)\alpha_0(t)^2\geq b'(t)+[2\alpha_1(t)-2\delta(t)\lambda(t)-C_2\delta(t)]b(t)\\
\nonumber\geq b'(t)+C_1b(t)\quad\text{(by \eqref{202401310811})}.    
\end{gather}
After integrating the last inequality with respect to the variable $t\in[t_0,s]$, we get
\begin{gather*}
    \frac{1}{2C_2}\int\limits_{t_0}^s\delta(t)\alpha_0(t)^2\,dt+b(t_0)
    \geq b(s)+C_1\int\limits_{t_0}^sb(t)\,dt,
\end{gather*}
which by \eqref{202401310812}, yields \eqref{202402020817}.

{\bf Step 2:} Show that
\begin{gather}\label{202402020906}
    \text{$x$ is bounded.}
\end{gather}
We observe
\begin{gather}
    \nonumber v_{x_*}''(t)+\alpha_1(t)v_{x_*}'(t)=2b(t)+\delta(t)\inner{y(t)-x(t)}{x(t)-x_*}\\
    \nonumber=2b(t)+\delta(t)\inner{y(t)-x_*}{x(t)-x_*}-2\delta(t)v_{x_*}(t)\\
    \label{202402020848}\leq 2b(t)+\frac{\delta(t)}{2}\|y(t)-x_*\|^2-\delta(t)v_{x_*}(t)\quad\text{(by the Cauchy-Schwarz inequality)}\\
    \nonumber\leq [2+\delta(t)\lambda(t)^2]b(t)+\delta(t)\alpha_0(t)^2+\delta(t)\lambda(t)v_{x_*}'(t)\quad\text{(by \eqref{202401310842})},
\end{gather}
and so
\begin{gather}
    \nonumber [2+\delta(t)\lambda(t)^2]b(t)+\delta(t)\alpha_0(t)^2\\
    \nonumber\geq v_{x_*}''(t)+\frac{d}{dt}\left\{[\alpha_1(t)-\delta(t)\lambda(t)]v_{x_*}(t)\right\}-[\alpha_1(t)-\delta(t)\lambda(t)]'v_{x_*}(t)\\
    \label{202402020830}\geq v_{x_*}''(t)+\frac{d}{dt}\left\{[\alpha_1(t)-\delta(t)\lambda(t)]v_{x_*}(t)\right\}\quad\text{(by \eqref{202401310810})}.
\end{gather}
We integrate the last inequality with respect to the variable $t\in[t_0,s]$
\begin{gather*}
    \int\limits_{t_0}^s[2+\delta(t)\lambda(t)^2]b(t)\,dt+\int\limits_{t_0}^s\delta(t)\alpha_0(t)^2\,dt
    +v_{x_*}'(t_0)+[\alpha_1(t_0)-\delta(t_0)\lambda(t_0)]v_{x_*}(t_0)\\
    \geq v_{x_*}'(s)+[\alpha_1(s)-\delta(s)\lambda(s)]v_{x_*}(s)\\
    \geq v_{x_*}'(s)+\frac{1}{2}C_1v_{x_*}(s)\quad\text{(by \eqref{202401310811})}.
\end{gather*}
Using \eqref{202402201529}, \eqref{202401311436}, \eqref{202402020817} and \eqref{202401310812}, there exists a constant $M_1=M_1(t_0)$ such that
\begin{gather*}
M_1\geq v_{x_*}'(s)+\frac{1}{2}C_1v_{x_*}(s),  
\end{gather*}
which gives \eqref{202402020906}.

With all preparation in place we can now prove the convergence of $x(t).$

(i) It results from \eqref{202401310822} that
\begin{gather*}
    \|x''(t)\|\leq\alpha_1(t)\|x'(t)\|+\delta(t)\|y(t)-x(t)\|\\
    \leq[\alpha_1(t)+\delta(t)\lambda(t)]\cdot\|x'(t)\|+\delta(t)\alpha_0(t)\quad\text{(by \eqref{202401310830})},
\end{gather*}
which gives
\begin{gather*}
\|x''(t)\|^2\leq[2(\alpha_1(t)+\delta(t)\lambda(t))^2+1]\cdot[b(t)+\delta(t)^2\alpha_0(t)^2]\\
\leq[2(\alpha_1(t)+\delta(t)\lambda(t))^2+1]\cdot\left[b(t)+\frac{2\alpha_1(t_0)}{C_2}\delta(t)\alpha_0(t)^2\right]\quad\text{(by \eqref{202401311436})}\\
\leq\text{constant}\cdot\left[b(t)+\frac{2\alpha_1(t_0)}{C_2}\delta(t)\alpha_0(t)^2\right]\quad\text{(by \eqref{202402201529}, \eqref{202401311436}, \eqref{202402090927})}\\
\in L^1\quad\text{(by \eqref{202402020817} and \eqref{202401310812})}.
\end{gather*}

(ii) The limit holds as we make use of Lemma \ref{lem-main-1}, together with the following
\begin{gather*}
    \frac{d}{dt}\|x'(t)\|^2=\inner{x''(t)}{x'(t)}\leq\frac{1}{2}(\|x''(t)\|^2+\|x'(t)\|^2)\in L^1.
\end{gather*}

(iii) It follows from \eqref{202402020830} that
\begin{gather}
\nonumber [2+\delta(t)\lambda(t)^2]b(t)+\delta(t)\alpha_0(t)^2
\geq v_{x_*}''(t)+\frac{d}{dt}\left\{[\alpha_1(t)-\delta(t)\lambda(t)]v_{x_*}(t)\right\}\\
\label{202401311512}=\frac{d}{dt}\{v_{x_*}'(t)+[\alpha_1(t)-\delta(t)\lambda(t)]v_{x_*}(t)\}.
\end{gather}
Inequalities \eqref{202401311512}, \eqref{202402020817} and \eqref{202401310812} together with Lemma \ref{lem-main-0} show that the limit 
\begin{gather*}
\lim\limits_{t\to\infty}\{v_{x_*}'(t)+[\alpha_1(t)-\delta(t)\lambda(t)]v_{x_*}(t)\}    
\end{gather*}
exists. Note that by (ii), the limit
\begin{gather*}
\lim\limits_{t\to\infty}[\alpha_1(t)-\delta(t)\lambda(t)]v_{x_*}(t)=\lim\limits_{t\to\infty}\{v_{x_*}'(t)+[\alpha_1(t)-\delta(t)\lambda(t)]v_{x_*}(t)\}    
\end{gather*}
exists, too. Condition \eqref{202401310810} shows that the function $\alpha_1(t)-\delta(t)\lambda(t)$ is decreasing and \eqref{202401310811} gives
\begin{gather*}
\alpha_1(t)-\delta(t)\lambda(t)\geq\frac{C_1}{2}>0.    
\end{gather*}
Consequently, there exists $\lim\limits_{t\to\infty}[\alpha_1(t)-\delta(t)\lambda(t)]>0$. From what have been proven, 
\begin{gather}\label{202402021548}
\text{$\lim\limits_{t\to\infty}\|x(t)-x_*\|^2$ exists for every $x_*\in\tap$}.
\end{gather}
Using \eqref{202402020906}, (ii), \eqref{202402081959} and the assumption that the operator $U$ is continuous, there is a constant $M_3=M_3(t_0)>0$ subject to the following
\begin{gather*}
    \|U(x(t)+\lambda(t)x'(t))\|\leq M_3\quad\forall t\geq t_0.
\end{gather*}
Denote
\begin{gather*}
    M_4\triangleq\frac{1}{\max\{1,M_3\}}.
\end{gather*}
By \eqref{202401310841}, we get
\begin{gather}\label{202401312036}
    \|y(t)-x_*\|^2
    \leq 2v_{x_*}(t)+2\lambda(t)^2b(t)+2\lambda(t)v_{x_*}'(t)+\alpha_0(t)^2\\
\nonumber-2M_4\alpha_0(t)\inner{U(x(t)+\lambda(t)x'(t))}{x(t)+\lambda(t)x'(t)-x_*}.
\end{gather}
Using \eqref{202401312036}, we estimate \eqref{202402020848} as follows
\begin{gather*}
    v_{x_*}''(t)+\alpha_1(t)v_{x_*}'(t)\\
    \leq [2+\delta(t)\lambda(t)^2]b(t)+\delta(t)\lambda(t)v_{x_*}'(t)+\delta(t)\alpha_0(t)^2\\
    -M_4\delta(t)\alpha_0(t)\inner{U(x(t)+\lambda(t)x'(t))}{x(t)+\lambda(t)x'(t)-x_*},
\end{gather*}
which is equivalent to saying that
{\small
\begin{gather*}
    [2+\delta(t)\lambda(t)^2]b(t)+\delta(t)\alpha_0(t)^2\\
    \geq v_{x_*}''(t)+[\alpha_1(t)-\delta(t)\lambda(t)]v_{x_*}'(t)
    +M_4\delta(t)\alpha_0(t)\inner{U(x(t)+\lambda(t)x'(t))}{x(t)+\lambda(t)x'(t)-x_*}\\
    \geq\dfrac{d}{dt}\{v_{x_*}'(t)+[\alpha_1(t)-\delta(t)\lambda(t)]v_{x_*}(t)\}
    +M_4\delta(t)\alpha_0(t)\inner{U(x(t)+\lambda(t)x'(t))}{x(t)+\lambda(t)x'(t)-x_*}.
\end{gather*}
}
Using \eqref{202402201529}, \eqref{202401311436}, \eqref{202402090927}, we have
{\small
\begin{gather*}
\text{constant}\cdot b(t)+\delta(t)\alpha_0(t)^2
    \geq[2+\delta(t)\lambda(t)^2]b(t)+\delta(t)\alpha_0(t)^2\\
    \geq\dfrac{d}{dt}\{v_{x_*}'(t)+[\alpha_1(t)-\delta(t)\lambda(t)]v_{x_*}(t)\}
    +M_4\delta(t)\alpha_0(t)\inner{U(x(t)+\lambda(t)x'(t))}{x(t)+\lambda(t)x'(t)-x_*}.
\end{gather*}
}
Integrating with respect to the variable $t\in[t_0,s]$,
{\small
\begin{gather*}
\text{constant}\cdot \int\limits_{t_0}^sb(t)\,dt+\int\limits_{t_0}^s\delta(t)\alpha_0(t)^2\,dt
+v_{x_*}'(t_0)+[\alpha_1(t_0)-\delta(t_0)\lambda(t_0)]v_{x_*}(t_0)\\
\geq v_{x_*}'(s)+[\alpha_1(s)-\delta(s)\lambda(s)]v_{x_*}(s)
+M_4\int\limits_{t_0}^s\delta(t)\alpha_0(t)\inner{U(x(t)+\lambda(t)x'(t))}{x(t)+\lambda(t)x'(t)-x_*}\,dt\\
\geq v_{x_*}'(s)
+M_4\int\limits_{t_0}^s\delta(t)\alpha_0(t)\inner{U(x(t)+\lambda(t)x'(t))}{x(t)+\lambda(t)x'(t)-x_*}\,dt\quad\text{(by \eqref{202401310811})}.
\end{gather*}
}
Use \eqref{202401310812}, \eqref{202402020817} and \eqref{202402020906} to get
\begin{gather*}
\int\limits_{t_0}^{\infty}\delta(t)\alpha_0(t)\inner{U(x(t)+\lambda(t)x'(t))}{x(t)+\lambda(t)x'(t)-x_*}\,dt<\infty,
\end{gather*}
which implies, by \eqref{202401310813}, that
\begin{gather*}
\inf\{\inner{U(x(t)+\lambda(t)x'(t))}{x(t)+\lambda(t)x'(t)-x_*}:t\geq t_0\}=0.   
\end{gather*}
Consequently, there is a sequence $(t_n)$ with the property
\begin{gather}
\label{202402021544}\lim\limits_{n\to\infty}\inner{U(x(t_n)+\lambda(t_n)x'(t_n))}{x(t_n)+\lambda(t_n)x'(t_n)-x_*}=0.
\end{gather}
It follows from \eqref{202402020906} that there exists a convergent subsequence $(x(t_{n_k}))$ of $(x(t_n))$. Denote
\begin{gather}\label{202402021550}
x_{\odot}\triangleq\lim\limits_{k\to\infty}x(t_{n_k}).    
\end{gather}
Letting $k\to\infty$ in \eqref{202402021544}, we get
\begin{gather*}
    \inner{U(x_{\odot})}{x_{\odot}-x_*}=0,
\end{gather*}
which gives $x_{\odot}\in\tap$ by Lemma \ref{lem-tap}. Hence, by \eqref{202402021548}, the limit $\lim\limits_{t\to\infty}\|x(t)-x_{\odot}\|$ exists and then by \eqref{202402021550} we obtain $\lim\limits_{t\to\infty}\|x(t)-x_{\odot}\|=0$.
\end{proof}

\subsubsection{Parameters choices}
Some examples are given to illustrate Assumption \ref{asu-coefficients}. 

\begin{prop}
Consider the dynamical system \eqref{202401310822}, where
\begin{gather*}
\begin{cases}
    \alpha_0(t)=\frac{1}{(t+1)^q},\quad\alpha_1(t)=h+\frac{1}{(t+1)^s},\\
    \\
    \delta(t)=\frac{1}{(t+1)^p},\quad\lambda(t)=0.
\end{cases}
\end{gather*}
Then $\lim\limits_{t\to\infty}x(t)=x_\star\in\tap$ if conditions \eqref{202410262114}-\eqref{202410262117} hold.
\begin{gather}
\label{202410262114}h>2.\\
\label{202410262115}0<s<\frac{1}{2}.\\
\label{202410262116}s<p<1.\\
\label{202410262117}\frac{1}{2}(1-p)<q\leq 1-p.
\end{gather}
\end{prop}
\begin{proof}
We have
\begin{gather*}
\alpha_1(t)^2+2\alpha_1'(t)=h^2+\frac{2h}{(t+1)^s}+\frac{1}{(t+1)^{2s}}-\frac{2s}{(t+1)^{s+1}}\\
>0+\frac{2h}{(t+1)^s}+0\quad\text{(by \eqref{202410262115})}\\
>\frac{4}{(t+1)^s}>4\delta(t)\quad\text{(by \eqref{202410262114} and \eqref{202410262116})},
\end{gather*}
which gives \eqref{202403022125}. Condition \eqref{202401310811} follows from the fact that
\begin{gather*}
2\alpha_1(t)-2\delta(t)\lambda(t)-2\delta(t)\\
=2h+\frac{2}{(t+1)^s}-0-\frac{2}{(t+1)^p}>2h\quad\text{(by \eqref{202410262116})},
\end{gather*}
where $C_1=2h$ and $C_2=2$. It can be observed that 
\begin{gather*}
\frac{d}{dt}[\alpha_1(t)-\delta(t)\lambda(t)]=-\frac{s}{(t+1)^{s+1}}<0\quad\text{(by \eqref{202410262115})},  
\end{gather*}
which proves \eqref{202401310810}. Condition \eqref{202401310812} is equivalent to saying that $p+2q>1$; meanwhile, condition \eqref{202401310813} is the fact that $p+q\leq 1$. Thus, we can make use of Theorem \ref{thm202402061556}.
\end{proof}

\begin{prop}\label{pp1}
Consider the dynamical system \eqref{202401310822}, where
\begin{gather*}
\begin{cases}
    \alpha_0(t)=\frac{1}{(t+1)^q},\quad\alpha_1(t)=h+\frac{1}{(t+1)^s},\\
    \\
    \delta(t)=u\,\text{(positive constant)},\quad\lambda(t)=0.
\end{cases}
\end{gather*}
Then $\lim\limits_{t\to\infty}x(t)=x_\star\in\tap$ if conditions \eqref{202410262114}-\eqref{202410262117} hold.
\begin{gather}
\label{202410262114a}h>2\sqrt{u}.\\
\label{202410262115a}0<s<\frac{1}{2}.\\
\label{202410262117a}\frac{1}{2}<q\leq 1.
\end{gather}
\end{prop}
\begin{proof}
We have
\begin{gather*}
\alpha_1(t)^2+2\alpha_1'(t)=h^2+\frac{2h}{(t+1)^s}+\frac{1}{(t+1)^{2s}}-\frac{2s}{(t+1)^{s+1}}\\
>h^2+\frac{2h}{(t+1)^s}+0\quad\text{(by \eqref{202410262115a})}\\
>h^2>4u\quad\text{(by \eqref{202410262114a})},
\end{gather*}
which gives \eqref{202403022125}. Condition \eqref{202401310811} follows from the fact that
\begin{gather*}
2\alpha_1(t)-2\delta(t)\lambda(t)-\frac{h}{u}\delta(t)\\
=2h+\frac{2}{(t+1)^s}-0-h>h,
\end{gather*}
where $C_1=\frac{h}{u}$ and $C_2=h$. It can be observed that 
\begin{gather*}
\frac{d}{dt}[\alpha_1(t)-\delta(t)\lambda(t)]=-\frac{s}{(t+1)^{s+1}}<0\quad\text{(by \eqref{202410262115a})},  
\end{gather*}
which proves \eqref{202401310810}. Condition \eqref{202401310812} is equivalent to saying that $2q>1$; meanwhile, condition \eqref{202401310813} is the fact that $q\leq 1$. Thus, we can make use of Theorem \ref{thm202402061556}.
\end{proof}

\section{Convergence analysis in discrete time}
In the section, we offer a convergence analysis for the dynamical system \eqref{202402061434}. Our analysis uses the following functions
\begin{gather*}
    v_{x_*}(n)\triangleq\|z(n)-x_*\|^2,\quad a(n)\triangleq\|z^{\dla}(n)\|^2,\quad c(n)\triangleq\|z^{\nabla}(n)\|^2.
\end{gather*}

\begin{lem}
Let $w(n)$ be the sequence given by \eqref{202402250850}. Then it holds that
\begin{gather}
\nonumber\|w(n)-x_*\|^2\\
\label{202401310841a2}\leq v_{x_*}(n)+\eta(n)v_{x_*}^{\nabla}(n)+\eta(n)[\eta(n)+1]c(n)+\beta_0(n)^2\\
\nonumber-\frac{2\beta_0(n)}{\max\{1,\|U(z(n)+\eta(n)z^{\nabla}(n))\|\}}\inner{U(z(n)+\eta(n)z^{\nabla}(n))}{z(n)+\eta(n)z^{\nabla}(n)-x_*}\\
\label{202401310842a2}\leq v_{x_*}(n)+\eta(n)v_{x_*}^{\nabla}(n)+\eta(n)[\eta(n)+1]c(n)+\beta_0(n)^2
\end{gather}
and
\begin{gather}
\label{202402092025}\|w(n)-z(n)\|\leq |\eta(n)|\cdot\|z^{\nabla}(n)\|+\beta_0(n).
\end{gather}
\end{lem}
\begin{proof}
We have
\begin{gather*}
    v_{x_*}^{\nabla}(n)=2\inner{z^{\nabla}(n)}{z(n)-x_*}-c(n),
\end{gather*}
and so
\begin{gather*}
    \|z(n)-x_*+\eta(n)z^{\nabla}(n)\|^2\\
    =\|z(n)-x_*\|^2+\eta(n)^2\|z^{\nabla}(n)\|^2+2\eta(n)\inner{z^{\nabla}(n)}{z(n)-x_*}\\
    =v_{x_*}(n)+\eta(n)v_{x_*}^{\nabla}(n)+\eta(n)[\eta(n)+1]c(n).
\end{gather*}
Since $z(n)\in\Omega$ by Proposition \ref{202402092014}, we have $P_\Omega(z(n))=z(n)$. Using arguments like as \eqref{202401310830}, \eqref{202401310841} and \eqref{202401310842}, we get
{\small
\begin{gather}
\nonumber\|w(n)-x_*\|^2\\
\nonumber\leq\left\|z(n)+\eta(n)z^{\nabla}(n)-x_*-\frac{\beta_0(n)}{\max\{1,\|U(z(n)+\eta(n)z^{\nabla}(n))\|\}}U(z(n)+\eta(n)z^{\nabla}(n))\right\|^2\\
\nonumber=\|z(n)+\eta(n)z^{\nabla}(n)-x_*\|^2+\left(\frac{\beta_0(n)}{\max\{1,\|U(z(n)+\eta(n)z^{\nabla}(n))\|\}}\right)^2\|U(z(n)+\eta(n)z^{\nabla}(n))\|^2\\
\nonumber-\frac{2\beta_0(n)}{\max\{1,\|U(z(n)+\eta(n)z^{\nabla}(n))\|\}}\inner{U(z(n)+\eta(n)z^{\nabla}(n))}{z(n)+\eta(n)z^{\nabla}(n)-x_*},
\end{gather}
}
which prove \eqref{202401310841a2} and \eqref{202401310842a2}. Proving \eqref{202402092025} is left to the reader as it is similar to \eqref{202401310830}.
\end{proof}

We study the dynamical system \eqref{202402061434}, under conditions \eqref{202402040925}-\eqref{202402032026}.
\begin{asu}\label{asu-dis-tim}
Assume that there exist constants $Q_2\in(0,1),Q_1>0$ such that
\begin{gather}
    \label{202402040925}[\beta_1(n)-1+\xi(n)][\beta_1(n)-1+(\eta(n)^2+Q_1)\xi(n)]\leq 1-Q_2\quad\forall n\in\Z_{\geq 0}.
\end{gather}
And assume that
\begin{gather}
\label{202402092106}-1\leq\eta(n)\leq 0\quad\forall n\in\Z_{\geq 0}.\\
\label{202404021032}\text{The sequence $\beta_1(n)-\xi(n)\eta(n)$ is decreasing}.\\
\label{202402041047}\sum_{j=0}^\infty\xi(j)\beta_0(j)^2<\infty.\\
\label{202402042040}\sum_{j=0}^\infty\xi(j)\beta_0(j)=\infty.\\
\label{202402032026}1\leq\beta_1(n)\leq\beta_1(n)+\xi(n)\leq 2\quad\forall n\in\Z_{\geq 0}.
\end{gather}
\end{asu}

\begin{rem}
Under Assumption \ref{asu-dis-tim}, we have some remarks. 
\begin{enumerate}
    \item[(i)] By condition \eqref{202402092106}, we have
\begin{gather}\label{202402112025}
z(n)+\eta(n)z^{\nabla}(n)=[1+\eta(n)]z(n)+(-\eta(n))z(n-1)\in\Omega.
\end{gather}
    \item[(ii)] Condition \eqref{202402040925} reveals
\begin{gather*}
    \beta_1(n)
    \geq\xi(n)[\beta_1(n)-1]+\beta_1(n)[\beta_1(n)-1]+\xi(n)[\xi(n)+\beta_1(n)-1][\eta(n)^2+Q_1]+Q_2\\
    \geq\xi(n)^2[\eta(n)^2+Q_1]+Q_2\quad\text{(as $\beta_1(n)\geq 1$)}\\
    >\xi(n)^2\eta(n)^2+Q_2,
\end{gather*}
which implies, as $\beta_1(n)\geq 1$, that
\begin{gather*}
\beta_1(n)^2>\xi(n)^2\eta(n)^2+Q_2.    
\end{gather*}
Thus, we get $\beta_1(n)>\sqrt{\xi(n)^2\eta(n)^2+Q_2}$ and
\begin{gather}
    \nonumber\beta_1(n)-\xi(n)\eta(n)>\frac{Q_2}{\sqrt{\xi(n)^2\eta(n)^2+Q_2}+\xi(n)\eta(n)}\\
    \label{202402102229}\geq 1-q\triangleq\frac{Q_2}{\sqrt{1+Q_2}}\quad\text{(as $\xi(n)\leq 2-\beta_1(n)\leq 1$)}.
\end{gather}
\item Conditions in \eqref{202402032026} help us use Proposition \ref{202402092014}.
\end{enumerate}
\end{rem}

\subsubsection{Theoretical analysis}
\begin{thm}
Consider the difference equation \eqref{202402061434}, under Assumptions \ref{asu-U} and \ref{asu-dis-tim}. Then it holds that $\lim\limits_{n\to\infty}z(n)=z_\star\in\tap.$
\end{thm}
\begin{proof}
For the sake of the reader, we divide the proof into the steps.

{\bf Step1:} Show that
\begin{gather}\label{2024020409376}
c\in\ell^\infty\cap\ell^1\Longrightarrow a\in\ell^\infty\cap\ell^1.    
\end{gather}
Indeed, we use Proposition \ref{rem20220120} to get
\begin{gather*}
c^{\dla}(n)+2\beta_1(n)c(n)\\
=2\inner{z^{\dla\nabla}(n)+\beta_1(n)z^{\dla}(n)}{z^{\nabla}(n)}+\|z^{\dla\nabla}(n)\|^2\\
=2\xi(n)\inner{w(n)-z(n)}{z^{\nabla}(n)}+\|-\beta_1(n)z^{\nabla}(n)+\xi(n)[w(n)-z(n)]\|^2\quad\text{(by \eqref{202402061434})}\\
=2\xi(n)[1-\beta_1(n)]\inner{w(n)-z(n)}{z^{\nabla}(n)}+\xi(n)^2\|w(n)-z(n)\|^2+\beta_1(n)^2c(n),
\end{gather*}
which implies, by \eqref{202402032026} and the Cauchy-Schwarz inequality, that
\begin{gather*}
c^{\dla}(n)+2\beta_1(n)c(n)\\
\leq\xi(n)[\xi(n)+\beta_1(n)-1]\|w(n)-z(n)\|^2+[\xi(n)(\beta_1(n)-1)+\beta_1(n)^2]c(n).
\end{gather*}
The inequality above is equivalent to the following
\begin{gather*}
c^{\dla}(n)+[2\beta_1(n)-\xi(n)(\beta_1(n)-1)-\beta_1(n)^2]c(n)\\
\leq\xi(n)[\xi(n)+\beta_1(n)-1]\cdot\|w(n)-z(n)\|^2\\
\leq\xi(n)[\xi(n)+\beta_1(n)-1]\cdot[-\eta(n)\|z^{\nabla}(n)\|+\beta_0(n)]^2\quad\text{(by \eqref{202402092025})}.
\end{gather*}
Using the the Cauchy-Schwarz inequality, we get
\begin{gather*}
c^{\dla}(n)+[2\beta_1(n)-\xi(n)(\beta_1(n)-1)-\beta_1(n)^2]c(n)\\
\leq\xi(n)[\xi(n)+\beta_1(n)-1]\cdot[Q_1c(n)+\beta_0(n)^2]\left(\frac{\eta(n)^2}{Q_1}+1\right),
\end{gather*}
which is equivalent to saying that
\begin{gather*}
c^{\dla}(n)+[2\beta_1(n)-\xi(n)(\beta_1(n)-1)-\beta_1(n)^2-\xi(n)(\xi(n)+\beta_1(n)-1)(\eta(n)^2+Q_1)]c(n)\\
\leq\xi(n)(\xi(n)+\beta_1(n)-1)\left(\frac{\eta(n)^2}{Q_1}+1\right)\beta_0(n)^2.
\end{gather*}
By \eqref{202402040925} and \eqref{202402032026}, we estimate
\begin{gather*}
c^{\dla}(n)+Q_2c(n)
\leq 2\xi(n)\left(\frac{\eta(n)^2}{Q_1}+1\right)\beta_0(n)^2\\
\leq 2\left(\frac{1}{Q_1}+1\right)\xi(n)\beta_0(n)^2,
\end{gather*}
where we use condition \eqref{202402092106}. Summing the line above from $n=1$ to $n=m-1$ yields
\begin{gather*}
c(m)-c(1)+Q_2\sum_{n=1}^{m-1}c(n)\leq 2\left(\frac{1}{Q_1}+1\right)\sum_{n=1}^{m-1}\xi(n)\beta_0(n)^2,   
\end{gather*}
which by \eqref{202402041047}, gives \eqref{2024020409376}.

Recall that a necessary condition for the convergence of a series is that the series terms must go to zero in the limit, and so
\begin{gather}\label{2024020416296}
    \lim\limits_{n\to\infty}\|z(n)-z(n-1)\|=0.
\end{gather}

{\bf Step 2:} Next, we prove that
\begin{gather}\label{2024020416366}
    \text{$z$ is bounded.}
\end{gather}
To that aim, we compute
\begin{gather*}
    v_{x_*}^{\dla\nabla}(n)+\beta_1(n)v_{x_*}^{\nabla}(n)=2\inner{z^{\dla\nabla}(n)+\beta_1(n)z^{\nabla}(n)}{z(n)-x_*}\\
    +2a(n)-\beta_1(n)c(n)-c^\dla(n),
\end{gather*}
which implies, by \eqref{202402061434}, that
\begin{gather*}
v_{x_*}^{\dla\nabla}(n)+\beta_1(n)v_{x_*}^{\nabla}(n)\\
\leq 2\xi(n)\inner{w(n)-z(n)}{z(n)-x_*}+2a(n)-\beta_1(n)c(n)-c^\dla(n)\\
= 2\xi(n)\inner{w(n)-x_*}{z(n)-x_*}-2\xi(n)v_{x_*}(n)+2a(n)-\beta_1(n)c(n)-c^\dla(n).
\end{gather*}
By the Cauchy-Schwarz inequality, we get
\begin{gather}
\nonumber v_{x_*}^{\dla\nabla}(n)+\beta_1(n)v_{x_*}^{\nabla}(n)\\
\label{2024020417246}\leq\xi(n)\|w(n)-x_*\|^2-\xi(n)v_{x_*}(n)+2a(n)-\beta_1(n)c(n)-c^\dla(n),
\end{gather}
which yields, by \eqref{202401310842a2} and \eqref{202402032026}, that
\begin{gather}
\nonumber v_{x_*}^{\dla\nabla}(n)+\beta_1(n)v_{x_*}^{\nabla}(n)\leq\xi(n)\eta(n)v_{x_*}^{\nabla}(n)+\xi(n)\beta_0(n)^2+2a(n)\\
\nonumber+[\xi(n)\eta(n)(\eta(n)+1)-\beta_1(n)]c(n)-c^\dla(n).
\end{gather}
We get
\begin{gather*}
v_{x_*}^{\dla\nabla}(n)+[\beta_1(n)-\xi(n)\eta(n)]v_{x_*}^{\nabla}(n)\\
\leq\xi(n)\beta_0(n)^2+2a(n)+[\xi(n)\eta(n)(\eta(n)+1)-\beta_1(n)]c(n)-c^\dla(n)\\
\leq\xi(n)\beta_0(n)^2+2a(n)-c^\dla(n)\quad\text{(by \eqref{202402092106})}.
\end{gather*}
For $n\geq 1$, we have
{\small
\begin{gather}
\nonumber v_{x_*}^{\dla\nabla}(n)+[\beta_1(n)-\xi(n)\eta(n)]v_{x_*}^{\nabla}(n)\\
\label{2024020420326}\geq v_{x_*}^{\dla\nabla}(n)+[\beta_1(n)-\xi(n)\eta(n)]v_{x_*}(n)-[\beta_1(n-1)-\xi(n-1)\eta(n-1)]v_{x_*}(n-1)\quad\text{(by \eqref{202404021032})}.
\end{gather}
}
Thus, we obtain
\begin{gather*}
v_{x_*}^{\dla\nabla}(n)+[\beta_1(n)-\xi(n)\eta(n)]v_{x_*}(n)-[\beta_1(n-1)-\xi(n-1)\eta(n-1)]v_{x_*}(n-1)\\
    \leq\xi(n)\beta_0(n)^2+2a(n)-c^\dla(n),
\end{gather*}
which implies, after summing from $n=\kappa$ to $n=m-1$, that
\begin{gather}
\nonumber v_{x_*}^{\nabla}(m)+[\beta_1(m-1)-\xi(m-1)\eta(m-1)]v_{x_*}(m-1)+c(m)\\
\nonumber\leq v_{x_*}^{\nabla}(\kappa)+[\beta_1(\kappa-1)-\xi(k-1)\eta(k-1)]v_{x_*}(\kappa-1)+c(\kappa)\\
\label{2024020416426}+\sum_{n=\kappa}^{m-1}[\xi(n)\beta_0(n)^2+2a(n)].
\end{gather}
In particular with $\kappa=1$, the inequality above and condition \eqref{202402041047} show 
\begin{gather*}
v_{x_*}^{\nabla}(m)+[\beta_1(m-1)-\xi(m-1)\eta(m-1)]v_{x_*}(m-1)+c(m)\\
\leq v_{x_*}^{\nabla}(1)+[\beta_1(0)-\xi(0)\eta(0)]v_{x_*}(0)+c(1)\\
+\sum_{n=1}^{m-1}[\xi(n)\beta_0(n)^2+2a(n)],
\end{gather*}
which implies, by \eqref{202402102229} and \eqref{202402041047}, that
\begin{gather}
\nonumber v_{x_*}^{\nabla}(m)+(1-q)v_{x_*}(m-1)\\
\label{202402111652}\leq K\triangleq v_{x_*}^{\nabla}(1)+[\beta_1(0)-\xi(0)\eta(0)]v_{x_*}(0)+c(1)+\sum_{n=0}^{\infty}[\xi(n)\beta_0(n)^2+2a(n)].
\end{gather}
Here, $q$ is defined in \eqref{202402102229} and it satisfies
\begin{gather*}
    0<1-q=\frac{Q_2}{\sqrt{1+Q_2}}<\sqrt{Q_2}<1.
\end{gather*}
It results from \eqref{202402111652} that
\begin{gather*}
    K\geq v_{x_*}^{\nabla}(m)+(1-q)v_{x_*}(m-1)=v_{x_*}(m)-qv_{x_*}(m-1),
\end{gather*}
and so,
\begin{gather*}
v_{x_*}(m)\leq qv_{x_*}(m-1)+K\leq q^2v_{x_*}(m-2)+K(q+1)\leq\cdots\\
\leq q^mv_{x_*}(0)+K(q^{m-1}+q^{m-2}+\cdots+q+1)<v_{x_*}(0)+\frac{K}{1-q}.
\end{gather*}
The boundedness of $z$ is proven.

{\bf Step 3:} Show that
\begin{gather}\label{2024020421086}
    \text{$\lim\limits_{m\to\infty}v_{x_*}(m)$ exists for every $x_*\in\tap$.}
\end{gather}

To see that, we combine \eqref{2024020416296} with \eqref{2024020416366} to get the following limits.
\begin{gather*}
    \lim\limits_{m\to\infty}v_{x_*}^{\dla}(m)=0=\lim\limits_{m\to\infty}c(m),\\
    \lim\limits_{\kappa\to\infty}v_{x_*}^{\dla}(\kappa)=0=\lim\limits_{\kappa\to\infty}c(\kappa).
\end{gather*}
Through letting $m,\kappa\to\infty$ in \eqref{2024020416426} and then using \eqref{202402041047} and the limits above, we obtain
\begin{gather*}
    \limsup\limits_{m\to\infty}[\beta_1(m-1)-\xi(m-1)\eta(m-1)]v_{x_*}(m-1)\\
    \leq\liminf\limits_{\kappa\to\infty}[\beta_1(\kappa-1)-\xi(\kappa-1)\eta(\kappa-1)]v_{x_*}(\kappa-1),
\end{gather*}
which gives 
\begin{gather*}
\limsup\limits_{m\to\infty}[\beta_1(m-1)-\xi(m-1)\eta(m-1)]v_{x_*}(m-1)\\
=\liminf\limits_{\kappa\to\infty}[\beta_1(\kappa-1)-\xi(\kappa-1)\eta(\kappa-1)]v_{x_*}(\kappa-1).    
\end{gather*}
The equality above means that there exists
\begin{gather*}
\lim\limits_{m\to\infty}[\beta_1(m)-\xi(m)\eta(m)]v_{x_*}(m).   
\end{gather*}
Note that inequality \eqref{202402102229} shows that $\lim\limits_{m\to\infty}[\beta_1(m)-\xi(m)\eta(m)]>0$. From what have been proven, we get \eqref{2024020421086}.

With all preparation in place we can now prove the convergence of $z(n)$. Using \eqref{2024020416366} and the assumption that the operator $U$ is continuous, there is a constant $C_3>0$ subject to the following
\begin{gather*}
    \|U(z(n)+\eta(n)z^{\nabla}(n))\|\leq C_3\quad\forall n\in\Z_{\geq 0}.
\end{gather*}
Denote
\begin{gather*}
    C_4\triangleq\frac{1}{\max\{1,C_3\}}.
\end{gather*}
Use \eqref{202401310841a2} to estimate \eqref{2024020417246} as follows
\begin{gather*}
\nonumber\xi(n)\beta_0(n)^2+2a(n)+[\xi(n)\eta(n)(\eta(n)+1)-\beta_1(n)]c(n)-c^\dla(n)\\
\geq v_{x_*}^{\dla\nabla}(n)+[\beta_1(n)-\xi(n)\eta(n)]v_{x_*}^{\nabla}(n)\\
+2C_4\xi(n)\beta_0(n)\inner{U(z(n)+\eta(n)z^{\nabla}(n))}{z(n)+\eta(n)z^{\nabla}(n)-x_*}\\
\geq v_{x_*}^{\dla\nabla}(n)+[\beta_1(n)-\xi(n)\eta(n)]v_{x_*}(n)-[\beta_1(n-1)-\xi(n-1)\eta(n-1)]v_{x_*}(n-1)\\
+2C_4\xi(n)\beta_0(n)\inner{U(z(n)+\eta(n)z^{\nabla}(n))}{z(n)+\eta(n)z^{\nabla}(n)-x_*}\quad\text{(by \eqref{2024020420326}).}
\end{gather*}
By \eqref{202402092106}, we estimate
\begin{gather*}
\xi(n)\beta_0(n)^2+2a(n)-c^\dla(n)\\
\geq\xi(n)\beta_0(n)^2+2a(n)+[\xi(n)\eta(n)(\eta(n)+1)-\beta_1(n)]c(n)-c^\dla(n).
\end{gather*}
Thus,
\begin{gather*}
\xi(n)\beta_0(n)^2+2a(n)-c^\dla(n)\\
\geq v_{x_*}^{\dla\nabla}(n)+[\beta_1(n)-\xi(n)\eta(n)]v_{x_*}(n)-[\beta_1(n-1)-\xi(n-1)\eta(n-1)]v_{x_*}(n-1)\\
+2C_4\xi(n)\beta_0(n)\inner{U(z(n)+\eta(n)z^{\nabla}(n))}{z(n)+\eta(n)z^{\nabla}(n)-x_*}.
\end{gather*}
Summing the last inequality from $n=3$ to $n=m-1$, we get
\begin{gather*}
    \sum_{n=3}^{m-1}[\xi(n)\beta_0(n)^2+2a(n)]-c(m)+c(3)\\
    \geq v_{x_*}^{\nabla}(m)-v_{x_*}^{\nabla}(3)+[\beta_1(m-1)-\xi(m-1)\eta(m-1)]v_{x_*}(m-1)-[\beta_1(2)-\xi(2)\eta(2)]v_{x_*}(2)\\
    +2C_4\sum_{n=3}^{m-1}\xi(n)\beta_0(n)\inner{U(z(n)+\eta(n)z^{\nabla}(n))}{z(n)+\eta(n)z^{\nabla}(n)-x_*},
\end{gather*}
which yields, by \eqref{202402041047} and \eqref{2024020409376}, that
\begin{gather*}
    \sum_{n=1}^\infty\xi(n)\beta_0(n)\inner{U(z(n)+\eta(n)z^{\nabla}(n))}{z(n)+\eta(n)z^{\nabla}(n)-x_*}<\infty.
\end{gather*}
Condition \eqref{202402042040} and the line above give
\begin{gather*}
    \inf\{\inner{U(z(n)+\eta(n)z^{\nabla}(n))}{z(n)+\eta(n)z^{\nabla}(n)-x_*}:n\in\Z_{\geq 0}\}=0.
\end{gather*}
Consequently, one can find $(n_k)$ subject to the following
\begin{gather}\label{202402042101}
    \lim\limits_{k\to\infty}\inner{U(z(n_k)+\eta(n_k)z^{\nabla}(n_k))}{z(n_k)+\eta(n_k)z^{\nabla}(n_k)-x_*}=0.
\end{gather}
By \eqref{2024020416366}, there is a convergent subsequence of $(z_{n_k})$, say $(z_{n_{k_j}})$. Denote
\begin{gather*}
    z_{\odot}\triangleq\lim\limits_{j\to\infty}z_{n_{k_j}}.
\end{gather*}
Letting $j\to\infty$ in \eqref{202402042101}, we get
\begin{gather*}
\inner{U(z_{\odot})}{z_{\odot}-x_*}=0,    
\end{gather*}
which implies, by Lemma \ref{lem-tap}, that $z_{\odot}\in\tap.$ Now we apply \eqref{2024020421086} to the case when $x_*=z_{\odot}$ and so the proof is complete.
\end{proof}

\subsection{Parameters choices}
We give examples illustrating Assumption \ref{asu-dis-tim}. Let $p,q,\delta$, $\theta,$ $\lambda,\omega$ be the parameters satisfying conditions \eqref{202402131042}, \eqref{202402131043}, \eqref{202402131044}, \eqref{202402131045}, \eqref{202402131046}, \eqref{202402131047}, respectively.
\begin{gather}
    \label{202402131042}0<p<1,\\
    \label{202402131043}\frac{1}{2}(1-p)<q\leq 1-p,\\
    \label{202402131044}\delta>0,\\
    \label{202402131045}\theta>0,\\
    \label{202402131046}\lambda>0,\\
    \label{202402131047}\omega>\max\{e^{\frac{\ln(\delta+1)}{p}},e^{\frac{\ln\theta}{\lambda}}\}.
\end{gather}

\begin{prop} \label{parameters}
 Consider the difference equation \eqref{202402061434}, where the coefficients are of the following forms
 \begin{gather*}
 \begin{cases}
\beta_0(n)=\frac{1}{(n+\omega)^q},\quad\beta_1(n)=1+\frac{\delta}{(n+\omega)^p},\\
\xi(n)=\frac{1}{(n+\omega)^p},\quad\eta(n)=\frac{-\theta}{(n+\omega)^{\lambda}}.
 \end{cases}
 \end{gather*}
Then $\lim\limits_{n\to\infty}z(n)=z_\star\in\tap$ if conditions \eqref{202402131042}-\eqref{202402131047} hold.
\end{prop}
\begin{proof}
Condition \eqref{202402040925} holds with the choice
    \begin{gather*}
        Q_1\triangleq1-\left(\frac{\theta}{\omega^\lambda}\right)^2,\\
        Q_2\triangleq 1-\frac{\delta+1}{\omega^p}.
    \end{gather*}
Indeed,
\begin{gather*}
[\beta_1(n)-1+\xi(n)][\beta_1(n)-1+(\eta(n)^2+Q_1)\xi(n)]\\
\leq\beta_1(n)-1+(\eta(n)^2+Q_1)\xi(n)\\
=\frac{\delta}{(n+\omega)^p}+\frac{1}{(n+\omega)^p}\left(\frac{\theta^2}{(n+\omega)^{2\lambda}}+Q_1\right)\\
\leq\frac{\delta}{\omega^p}+\frac{1}{\omega^p}\left(\frac{\theta^2}{\omega^{2\lambda}}+Q_1\right)=1-Q_2.
\end{gather*}
\end{proof}

\section{Numerical experiments}\label{sec-exam}
In this section, we present some examples to illustrate the effectiveness of our algorithms. These examples were tested using Matlab software, running on a PC with CPU i5 10400 and 16Gb RAM.

Consider the problem \ref{proVI} with 
$$
\Omega=\{ x\in \mathbb R^3: \|x\| \leq 1  \}
$$
and $U: \mathbb R^3 \to \mathbb R^3$, $U(x)=Ax$ for all $x \in \mathbb R^3$, where
$
A= \begin{pmatrix}
1&-2&1\\
3&1&3\\
1&-2&1\\
\end{pmatrix}.
$
It is easy seen that $U$ is paramonotone but is not strictly monotone. We 
applied Algorithm \eqref{202401310822} and \eqref{202402061434}   to solve \ref{proVI}. 
\begin{itemize}
    \item  In Algorithm \eqref{202401310822}, we choose the parameters as in Proposition \ref{pp1}. We fix $u=1, s=0.35, q=0.71$ and adjust the parameter $h$. Choose $x(0)=(1; 0; 0)^T$ and $x'(0)= (-0.75;0.75;0)^T$. The results are presented in Figure  \ref{lientuc}.
    \item In Algorithm \eqref{202402061434}, we choose the parameters as in Proposition \ref{parameters} with different $p, q, \delta,\theta, \lambda$. Set $\omega=\max\left\{e^{\frac{\ln(\delta+1)}{p}},e^{\frac{\ln\theta}{\lambda}}\right\}+1$ and   $z(0)=(1;0;0)^T$, $z(1)=(0;1;0)^T$. The results are presented in Figure  \ref{viduroirac}.
\end{itemize}
\begin{figure}
    \centering
    \includegraphics[width=0.9\linewidth]{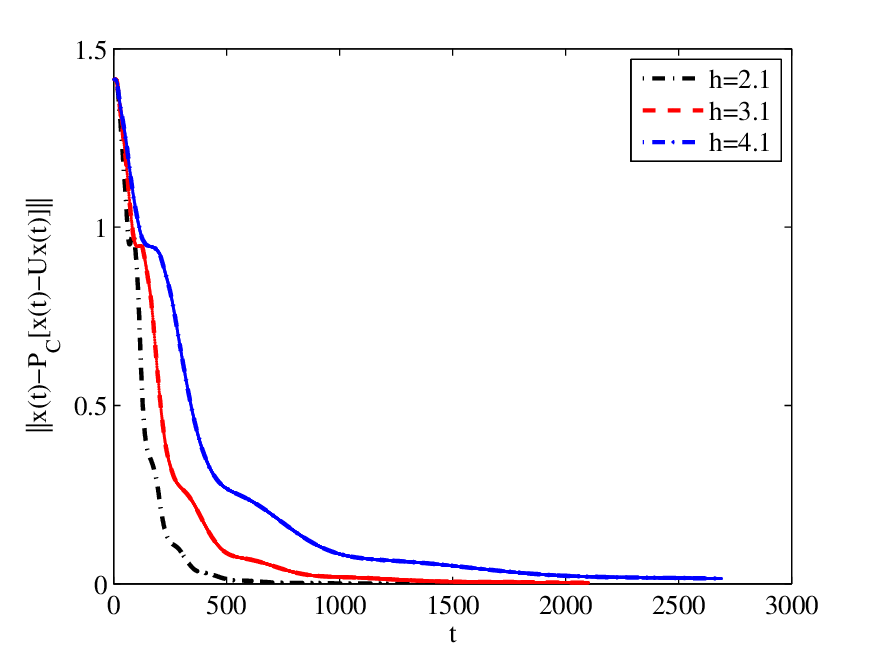}
    \caption{Performance of Algorithm \eqref{202401310822} with different parameter}
    \label{lientuc}
\end{figure} 

\begin{figure}
    \centering
    \includegraphics[width=0.9\linewidth]{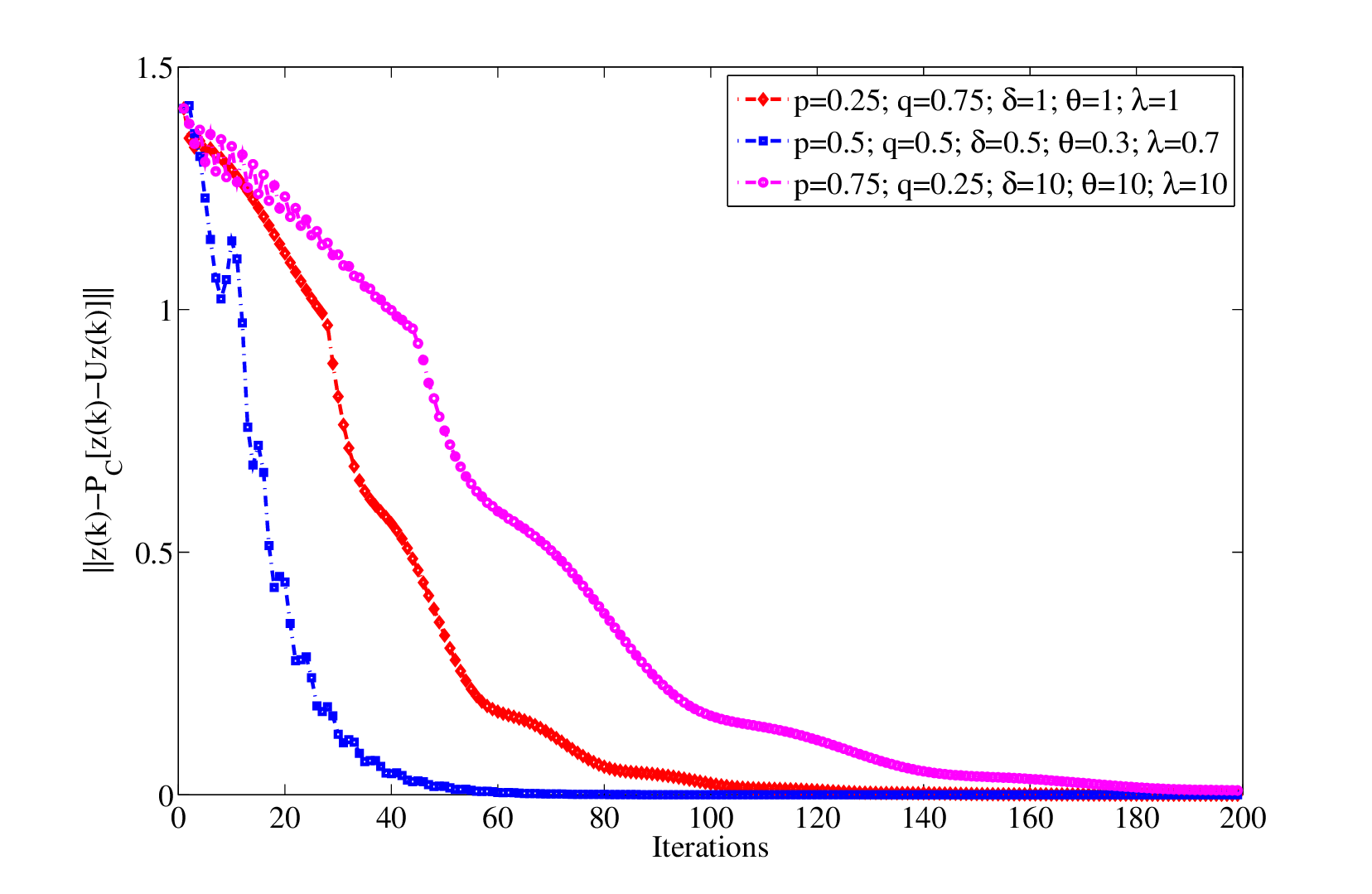}
    \caption{Performance of Algorithm \eqref{202402061434} with different parameters}
    \label{viduroirac}
\end{figure}
Next, we compare our algorithm \eqref{202402061434}  with the Direct method (\texttt{Dir.Method}) (Algorithm 1 in \cite{BelloIusem}) 
\begin{equation}\label{DR}
    \begin{cases}
        x(0) \in \Omega,\\
        x(n+1)=P_\Omega\left( x(n) -\frac{\beta_0(n)}{\max\{1; \| U(x(n)) \|  \}}U(x(n)) \right).\tag{\texttt{Dir.Method}}
    \end{cases}
\end{equation}
It was proven in \cite{BelloIusem} that the trajectory generated by \eqref{DR} converges to an element $z_\star\in\tap$ under the assumption that
\begin{gather}\label{20241024}
\sum\limits_{j=0}^\infty\beta_0(j)^2<\infty\quad\text{and}\quad\sum\limits_{j=0}^\infty\beta_0(j)=\infty.    
\end{gather}
Due to \eqref{20241024}, the \texttt{Dir.Method} uses the step sizes $ \beta_0(n)=\frac 1 {n^\tau} $, with $\tau \in (0.5;1]$. In our algorithm, the step size is taken as $\beta_0(n)=\frac{1}{(n+\omega)^q}$, where the range of the exponent $q$ is wider than $\tau$ of the \texttt{Dir.Method}; in details, $q\in(0;1)$. In the numerical examples, we fix $p=q=0.5$ and adjust the parameters $\delta, \theta, \lambda.$  Let  $\omega =\max\{(\delta+1)^{1/p},\theta^{1/\lambda}\}+1$. The comparison are presented in Figure \ref{fig1}. As we can see, in all cases the new algorithm is more efficient than the existing one. \\
\begin{figure}
    \centering
    \includegraphics[width=0.9\linewidth]{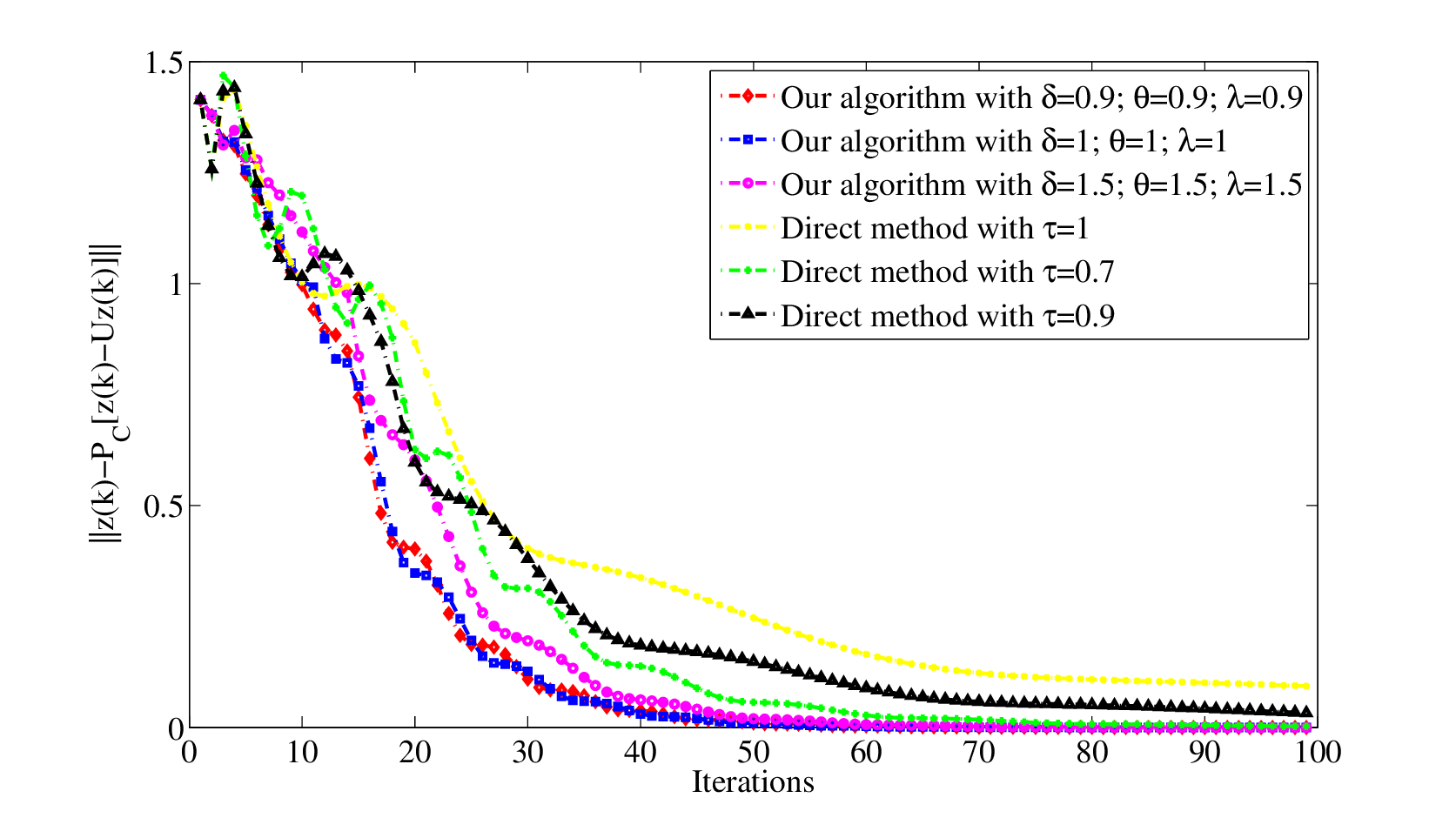}
    \caption{Comparison results of Algorithm \eqref{202402061434} and \texttt{Dir.Method}}
    \label{fig1}
\end{figure}

\section{Conclusion}
We have proposed a second order dynamical system for solving paramonotone VIs in a finite-dimensional space and studied its global convergence. A discretization of this dynamical system gives rise to an inertial iterative projection-type algorithm for which we establish the convergence of the iterations. Some numerical examples are given to confirm the theoretical results and illustrate the effectiveness of the proposed algorithms. The topic for future research includes second order dynamical systems enhanced by the vanishing damping coefficient for solving paramonotone VIs.

\section*{Data Availability}
Data sets generated during the current study are available from the corresponding author on reasonable request.

\section*{Competing interests}
The authors declare that they have no competing interests.

\nocite{*}

\begin{thebibliography}{99}
  \bibitem{Abbas}	B. Abbas, H. Attouch,  Benar F. Svaiter: Newton-like dynamics and forward-backward methods for  	structured monotone inclusions in Hilbert spaces. J. Optim. Theory Appl., 161, 331-360 (2014)
  \bibitem{Alecsa} C.D. Alecsa, S.C. Laszlo, T. Pinta: An Extension of the Second Order Dynamical System that Models Nesterov's Convex Gradient Method. Appl. Math Optim., 84, 1687-1716 (2021)

\bibitem{AnhVinhDOI} P.K. Anh, N.T. Vinh: Self-adaptive gradient projection algorithms for variational inequalities involving non-Lipschitz continuous operators. Numer. Algorithms, 81, 983-1001 (2019)

\bibitem{AA} H. Attouch, F. Alvarez: The heavy ball with friction dynamical system for convex constrained minimization problems, in: Optimization (Namur, 1998), Lecture Notes in Economics and Mathematical Systems 481, Springer, Berlin, 25-35, (2000)

\bibitem{ACR} H. Attouch, Z. Chbani,  H. Riahi: Combining fast inertial dynamics for convex optimization with Tikhonov regularization. J. Math. Anal. Appl., 457, 1065-1094 (2018)

\bibitem{zbMATH07194541} H. Attouch, Z. Chbani, H. Riahi: Rate of convergence of the Nesterov accelerated gradient method in the subcritical case $\alpha\leq 3$. ESAIM, Control Optim. Calc. Var., 25, 1292-8119 (2019)

\bibitem{Bao} T.Q. Bao, P.Q. Khanh: A projection-type algorithm for pseudomonotone nonlipschitzian multi-valued variational inequalities. Nonconvex Optim. Appl., 77, 113-129 (2005)

\bibitem{BelloIusem} J.Y. Bello Cruz, A.N. Iusem: Convergence of direct methods for paramonotone variational inequalities. Comput. Optim. Appl., 46, 247-263 (2010)
\bibitem{Bot} R. I. Bot, E. R. Csetnek: Second order forward-backward dynamical systems for monotone inclusion problems. SIAM J. Control Optim., 54, 1423-1443 (2016)
\bibitem{zbMATH06522738} R. I. Bot, E. R. Csetnek: Approaching the solving of constrained variational inequalities via penalty term-based dynamical systems. J. Math. Anal. Appl., 435, 1688-1700 (2016)
\bibitem{CensorGibali} Y. Censor, A. Gibali, S. Reich: The subgradient extragradient method for solving variational inequalities in Hilbert space. J. Optim. Theory and Appl., 148, 318-335 (2011)
\bibitem{zbMATH07344788} E. R. Csetnek: Continuous dynamics related to monotone inclusions and non-smooth optimization problems. Set-Valued Var. Anal., 28(4), 611-642 (2020)
\bibitem{Cruz} J.Y. Bello Cruz, R.D. Mill\'{a}n: A direct splitting method for nonsmooth variational inequalities. J. Optim. Theory Appl. 161, 729-737 (2014)
\bibitem{Facchinei} F. Facchinei, J.-S. Pang: Finite-dimensional variational inequalities and complementarity problems. Springer, New York (2003)
\bibitem{HaiOPTL}T.N. Hai: On gradient projection methods for strongly pseudomonotone variational inequalities without Lipschitz continuity.  Optim. Lett. 14, 1177-1191 (2020)
\bibitem{HaiDS} T.N. Hai: Dynamical systems for solving variational inequalities. J. Dyn. Control Syst., 28(4), 681-696, (2022)
\bibitem{Iiduka1}H. Iiduka: Fixed point optimization algorithm and its application to power control in CDMA data networks. Math. Program. 133, 227-242 (2012)
\bibitem{Iiduka2}H. Iiduka, I. Yamada: An ergodic algorithm for the power-control games for CDMA data networks. J. Math. Model. Algorithms 8, 1-18 (2009)
\bibitem{KhanhVuong} P.D. Khanh, P.T. Vuong: Modified projection method for strongly pseudomonotone variational inequalities. J. Global Optim. 58, 341-350 (2014)
\bibitem{Kinderlehrer} D. Kinderlehrer, G. Stampacchia: An introduction to variational inequalities and their applications. Academic Press, New York (1980)
\bibitem{MalitskySemenov} Y.V. Malitsky, V.V. Semenov: An extragradient algorithm for monotone variational inequalities,. Cybernetics and Systems Anal., 50, 271-277 (2014)
\bibitem{Polyak} B.T. Polyak: Some methods of speeding up the convergence of iteration methods. U.S.S.R. Comput. Math. Math. Phys. 4,  1-17 (1964)
\bibitem{Quoc1} T.D. Quoc, L.D. Muu, V.H. Nguyen: Extragradient algorithms extended to equilibrium problems. Optimization 57, 749-776 (2008)
\bibitem{SolodovSvaiter}M.V. Solodov, B.F. Svaiter: A new projection method for monotone variational inequality problems. SIAM J. Control Optim. 37, 765-776 (1999)
\bibitem{VuongSIAM} P.T. Vuong: A second order dynamical system and its discretization for strongly pseudo-monotone variational inequalities. SIAM J. Control Optim., 59(4), 2875-2897 (2021)
\end{thebibliography}

\end{document}